\documentclass[12pt]{amsart}
\usepackage{amsmath,amssymb,amsbsy,amsfonts,latexsym,amsopn,amstext,cite,
                                               amsxtra,euscript,amscd,bm,mathabx}
\usepackage{url}
\usepackage[colorlinks,linkcolor=blue,anchorcolor=blue,citecolor=blue,backref=page]{hyperref}
\usepackage{color}
\usepackage{graphics,epsfig}
\usepackage{graphicx}
\usepackage{float} 
\usepackage[english]{babel} 
\usepackage{mathtools}
\usepackage{todonotes}
\usepackage{url}
\usepackage[colorlinks,linkcolor=blue,anchorcolor=blue,citecolor=blue,backref=page]{hyperref}

\def\J{{\mathrm{Jac}}}

\usepackage[norefs,nocites]{refcheck}

\hypersetup{breaklinks=true}

\usepackage[norefs,nocites]{refcheck}
\usepackage[english]{babel}
\begin{document}

\newtheorem{thm}{Theorem}
\newtheorem{lem}[thm]{Lemma}
\newtheorem{claim}[thm]{Claim}
\newtheorem{cor}[thm]{Corollary}
\newtheorem{prop}[thm]{Proposition} 
\newtheorem{definition}[thm]{Definition}
\newtheorem{rem}[thm]{Remark} 
\newtheorem{question}[thm]{Open Question}
\newtheorem{qn}[thm]{Question}
\newtheorem{conj}[thm]{Conjecture}
\newtheorem{prob}{Problem}

\newcommand{\GL}{\operatorname{GL}}
\newcommand{\SL}{\operatorname{SL}}
\newcommand{\lcm}{\operatorname{lcm}}
\newcommand{\ord}{\operatorname{ord}}
\newcommand{\Op}{\operatorname{Op}}
\newcommand{\Tr}{\operatorname{Tr}}
\newcommand{\Nm}{\operatorname{Nm}}

\numberwithin{equation}{section}
\numberwithin{thm}{section}
\numberwithin{table}{section}

\def\vol {{\mathrm{vol\,}}}
\def\squareforqed{\hbox{\rlap{$\sqcap$}$\sqcup$}}
\def\qed{\ifmmode\squareforqed\else{\unskip\nobreak\hfil
\penalty50\hskip1em\null\nobreak\hfil\squareforqed
\parfillskip=0pt\finalhyphendemerits=0\endgraf}\fi}

\def \balpha{\bm{\alpha}}
\def \bbeta{\bm{\beta}}
\def \bgamma{\bm{\gamma}}
\def \blambda{\bm{\lambda}}
\def \bchi{\bm{\chi}}
\def \bphi{\bm{\varphi}}
\def \bpsi{\bm{\psi}}
\def \bomega{\bm{\omega}}
\def \btheta{\bm{\vartheta}}
\def \ochi{\overline{\chi}}
\def \h{\widehat h}

\def\ovG {\overline{\Gamma}}

\def\uu {\mathbf{u}} 
\def\uv {\mathbf{v}} 
\def\uk {\mathbf{k}} 
\def\ut {\mathbf{t}}

\def\eps{\varepsilon}

\newcommand{\bfxi}{{\boldsymbol{\xi}}}
\newcommand{\bfrho}{{\boldsymbol{\rho}}}

\def\Kab{\sfK_\psi(a,b)}
\def\Kuv{\sfK_\psi(u,v)}
\def\SaUV{\cS_\psi(\balpha;\cU,\cV)}
\def\SaAV{\cS_\psi(\balpha;\cA,\cV)}

\def\SUV{\cS_\psi(\cU,\cV)}
\def\SAB{\cS_\psi(\cA,\cB)}

\def\Kmnp{\sfK_p(m,n)}

\def\KKap{\cH_p(a)}
\def\KKaq{\cH_q(a)}
\def\KKmnp{\cH_p(m,n)}
\def\KKmnq{\cH_q(m,n)}

\def\Klmnp{\sfK_p(\ell, m,n)}
\def\Klmnq{\sfK_q(\ell, m,n)}

\def \SALMNq {\cS_q(\balpha;\cL,\cI,\cJ)}
\def \SALMNp {\cS_p(\balpha;\cL,\cI,\cJ)}

\def \SACXMQX {\fS(\balpha,\bzeta, \bxi; M,Q,X)}

\def\SAMJp{\cS_p(\balpha;\cM,\cJ)}
\def\SAMJq{\cS_q(\balpha;\cM,\cJ)}
\def\SAqMJq{\cS_q(\balpha_q;\cM,\cJ)}
\def\SAJq{\cS_q(\balpha;\cJ)}
\def\SAqJq{\cS_q(\balpha_q;\cJ)}
\def\SAIJp{\cS_p(\balpha;\cI,\cJ)}
\def\SAIJq{\cS_q(\balpha;\cI,\cJ)}

\def\RIJp{\cR_p(\cI,\cJ)}
\def\RIJq{\cR_q(\cI,\cJ)}

\def\TWXJp{\cT_p(\bomega;\cX,\cJ)}
\def\TWXJq{\cT_q(\bomega;\cX,\cJ)}
\def\TWpXJp{\cT_p(\bomega_p;\cX,\cJ)}
\def\TWqXJq{\cT_q(\bomega_q;\cX,\cJ)}
\def\TWJq{\cT_q(\bomega;\cJ)}
\def\TWqJq{\cT_q(\bomega_q;\cJ)}

 \def \xbar{\overline x}
  \def \ybar{\overline y}

\def\cA{{\mathcal A}}
\def\cB{{\mathcal B}}
\def\cC{{\mathcal C}}
\def\cD{{\mathcal D}}
\def\cE{{\mathcal E}}
\def\cF{{\mathcal F}}
\def\cG{{\mathcal G}}
\def\cH{{\mathcal H}}
\def\cI{{\mathcal I}}
\def\cJ{{\mathcal J}}
\def\cK{{\mathcal K}}
\def\cL{{\mathcal L}}
\def\cM{{\mathcal M}}
\def\cN{{\mathcal N}}
\def\cO{{\mathcal O}}
\def\cP{{\mathcal P}}
\def\cQ{{\mathcal Q}}
\def\cR{{\mathcal R}}
\def\cS{{\mathcal S}}
\def\cT{{\mathcal T}}
\def\cU{{\mathcal U}}
\def\cV{{\mathcal V}}
\def\cW{{\mathcal W}}
\def\cX{{\mathcal X}}
\def\cY{{\mathcal Y}}
\def\cZ{{\mathcal Z}}
\def\Ker{{\mathrm{Ker}}}
\def\g{{\mathrm{gcd}}}
\def\rad{{\mathrm{rad}}}
\def\rank{{\mathrm{rank}}}

\def\Pse {\mathsf E}
\def\Psf {\mathsf P}
\def\Qsf {\mathsf Q}
\def\fsf {\mathsf f}
\def\ksf {\mathsf k}
\def\hsf {\mathsf h}
\def\isf{\mathsf O}

\def\Dsf {\mathsf D}

\def\NmQR{N(m;Q,R)}
\def\VmQR{\cV(m;Q,R)}

\def\Xm{\cX_m}

\def \A {{ \mathbb A}}
\def \B {{ \mathbb B}}
\def \C {{ \mathbb C}}
\def \N {{ \mathbb N}}
\def \F {{ \mathbb F}}
\def \G {{\mathbb G}}
\def \L {{\mathbb L}}
\def \K {{\mathbb K}}
\def \PP {{\mathbb P}}
\def \Q {{\mathbb Q}}
\def \R {{\mathbb R}}
\def \Z {{\mathbb Z}}
\def \fS{\mathfrak S}

\def\e{{\mathbf{\,e}}}
\def\ep{{\mathbf{\,e}}_p}
\def\eq{{\mathbf{\,e}}_q}
\def\er{{\mathbf{\,e}}_R}

\def\\{\cr}
\def\({\left(}
\def\){\right)}
\def\fl#1{\left\lfloor#1\right\rfloor}
\def\rf#1{\left\lceil#1\right\rceil}

\def\Tr{{\mathrm{Tr}}}
\def\Nm{{\mathrm{Nm}}}
\def\Im{{\mathrm{Im}}}

\def \oF {\overline \F}

\newcommand{\pfrac}[2]{{\left(\frac{#1}{#2}\right)}}

\def \Prob{{\mathrm {}}}
\def\e{\mathbf{e}}
\def\ep{{\mathbf{\,e}}_p}
\def\epp{{\mathbf{\,e}}_{p^2}}
\def\em{{\mathbf{\,e}}_m}

\def\Res{\mathrm{Res}}
\def\Orb{\mathrm{Orb}}

\def\vec#1{\mathbf{#1}}
\def \va{\vec{a}}
\def \vb{\vec{b}}
\def \vm{\vec{m}}
\def \vu{\vec{u}}
\def \vv{\vec{v}}
\def \vx{\vec{x}}
\def \vy{\vec{y}}
\def \vz{\vec{z}}
\def\flp#1{{\left\langle#1\right\rangle}_p}
\def\T {\mathsf {T}}

\def\sfG {\mathsf {G}}
\def\sfK {\mathsf {K}}

\def\mand{\qquad\mbox{and}\qquad}

\title[Product of denominators of elliptic curves]
{Perfect powers in the product of denominators of elliptic curves}

\author{Subham Bhakta}
\address{School of Mathematics and Statistics, University of New South Wales, Sydney, NSW 2052, Australia.} 
\email{subham.bhakta@unsw.edu.au}

\dedicatory{Dedicated to the 70th birthday of Igor E. Shparlinski}

\begin{abstract}
We use sieving arguments to estimate the frequency of $s$-tuples of rational points 
$$(P_1,\dots,P_s)\in E_1(\Q)\times\cdots\times E_s(\Q),$$ where $E_1,\dots,E_s$ are (not necessarily distinct) elliptic curves over $\Q$, for which the product of their denominators is a perfect $\ell$th power for a fixed prime $\ell$. We consider two settings: one in which the points are of the form $n_iP_i+Q_i$ with $n_i$ ranging over an interval, and another in which we take arbitrary points of bounded canonical height. In the special case where all $Q_i$ are the points at infinity, we also obtain better estimates by using a version of the elliptic sieve with elliptic divisibility sequences. Consequently, we derive analogues of these results for various rational functions, providing elliptic analogues of R.~de la Bretèche, P.~Kurlberg and I.~E.~Shparlinski (2021).
\end{abstract}

\maketitle
\tableofcontents

\section{Introduction}
The study of perfect powers occurring among the terms, or among the products of terms, of various recurrence sequences has a long and rich history in Diophantine number theory. A classical result of Erdős and Selfridge~\cite{ES75} states that the Diophantine equation  
\[
n(n+1)\cdots(n+s-1)=m^{\ell},
\]
admits no solutions in integers $n\ge 0$ and $s,\ell\ge 2$. Of course, to study such questions, it is always enough to assume that $\ell$ is a prime. Thus, throughout the paper, the exponent $\ell$ is reserved as a prime. The related questions for products of the terms of an arithmetic progression have also been extensively studied; see for example \cite{B23,BBGH06,G13,HTT09}, and the references therein.

The case in which the factors arise from terms of linear recurrence sequences has also been extensively studied. In the special case of a single term, that is determining whether a term of a linear recurrence sequence can be a perfect power, the problem has received considerable attention. Notable contributions include those of Pethő \cite{P82}, Shorey and Stewart \cite{SS87}, Corvaja and Zannier~\cite{CZ02}, and Fuchs and Tichy \cite{FT03}. To the best of our knowledge, the most general result in this direction, building on several previous works, is due to Bugeaud and Kaneko~\cite{BK19}.

On the other hand, Luca and Shorey \cite{LS05} have obtained an effective upper bound for the solutions of equations in which a product of the terms of a Lucas sequence is a perfect power. This has been extended by Bravo, Das, Guzmán, and Laishram \cite{BDGL15} to certain special cases, specifically for the Pell and Pell-Lucas sequences, where they explicitly determined all the solutions. For further historical background and additional developments, we refer the reader to these works and the references therein.

We focus on Diophantine equations of the type
\[
f(n_1)\cdots f(n_s)=m^{\ell},
\]
where $f(n)$ is an arithmetic function taking integer or rational values, and each $n_i$ is an integer in an interval $(M,M+N]$. In the natural setting $f(n)=n$, such equations have received considerable attention. In particular, the case $s=3$ and $\ell=3$ has been studied by de la Bret\`eche~\cite{B98}, Fouvry~\cite{F98}, and Heath-Brown and Moroz~\cite{BM99}. Further work includes Tolev~\cite{Tol11} for $s=2$ and $\ell>2$, and more recently de la Bret\`eche, Kurlberg, and Shparlinski~\cite{BKS21} in the general case. In this paper, we turn to \textit{sparse} sequences $\(f(n)\)_{n\ge 1}$ arising from certain nonlinear recursions, notably \textit{elliptic divisibility sequences}, as well as the sequences associated with rational points on elliptic curves.

\subsection{Set-up} 
For an  elliptic curve $E$ given by a short Weierstrass equation 
\begin{equation}\label{eq:Weier}
y^2   = x^3 + a x + b
\end{equation}
with integral coefficients $a$ and $b$, we denote by $E(\Q)$ the group of rational points of $E$, and $O$ denotes the point at infinity, see~\cite{Silv-Book} for background. We can write any point $P \in E(\Q)$, in the lowest form 
\begin{equation}\label{eqn:P}
P =(x(P),y(P))=\left(\frac{a_P}{d_P^2},\frac{b_P}{d_P^3}\right),
\end{equation}
where $d_P \in\N$, and $\gcd(a_Pb_P , d_P)=1$. 

Everest, Reynolds, and Stevens \cite{ERS} have shown that, for any fixed $\ell$, there are only finitely many perfect $\ell$th power among the denominators $d_P$. They have also shown the same uniformly over $\ell$, assuming that the $abc$-conjecture holds. Unconditionally, first Reynolds~\cite{R12} and then Alfaraj~\cite{A23} have studied this uniformly over the exponent $\ell$, for certain elliptic curves. More recently, Nowroozi and Siksek \cite{NS} further extended these results over a much broader family of elliptic curves.

We are interested in the product of the denominators; in particular, we study how often such products are perfect powers. Hajdu, Laishram and Szikszai in \cite[Theorem 1]{HLS} have studied this over denominators of multiples of a single fixed point. That is, the terms are from a fixed \textit{elliptic divisibility sequence}. More precisely, they showed that there are finitely many perfect $\ell$th powers among the products
$$d_{nP}d_{(n+h)P} \cdots d_{(n+(s-1)h)P},\quad n,s,h \in \mathbb{N},$$
and in Theorem~\ref{thm:singlevar}, we record an analogue for arbitrary shifts.

However, the arguments of Hajdu, Laishram and Szikszai do not work when the points are of the form $nP+Q$, where $P,Q$ are fixed rational points, and $n$ ranges over an interval. One of the key obstacles is that, we no longer have the required $p$-adic properties (as in \cite[Lemma~2.1]{HLS}) of these points. Similarly, the same arguments also do not work when we consider the products over all $s$-tuples of rational points. Recently, the products of denominators of rational points of this kind have been considered by B{\'e}rczes, Hajdu, Ostafe, Shparlinski and the present author in \cite{BBHOS26}, and they have shown that the denominators of such points are \textit{rarely} multiplicatively dependent. Building upon \cite{BBHOS26} and \cite{HLS}, we study the frequency of tuples of rational points, whose product of denominators is multiplicatively related to a perfect $\ell$th power. For any prime $\ell$, we denote 
$$\cN_{\ell}=\{n^{\ell}:n\in \N\}\quad \mathrm{and}\quad \cQ_{\ell}=\{x^{\ell}:x\in \Q\}.$$

Assume that we are given $s$ not necessarily distinct elliptic curves $E_1, \ldots, E_s$ as in \eqref{eq:Weier} over $\Q$ of positive ranks, and a prime $\ell$. For fixed $s$-tuples 
\begin{equation}
\label{eq: P Q}
\Psf=(P_1, \ldots, P_s) \mand \Qsf=(Q_1, \ldots, Q_s),  
\end{equation}
of non-torsion points $P_i\in  E_i(\Q)$ and arbitrary 
points $Q_i \in E_i(\Q)$, $i =1, \ldots, s$, we estimate the size of the following set
\begin{align*}
&\cD_{\ell,\Psf,\Qsf}(M,N)\\
&\quad=\{ (n_1, \ldots, n_s)\in   (M,M+N]^s:\\
&\qquad\quad\quad d_{n_1P_1 + Q_1}^{k_1} \cdots d_{n_sP_s+Q_s}^{k_s} \in \cN_{\ell},~\text{for some }(k_1,\ldots,k_s)\in [1,\ell-1]^s \}.
\end{align*}

Furthermore, we denote by $\h(P)$ the {\it canonical height\/} of a rational point $P\in E(\Q)$, see for instance~\cite[Section~VIII.9]{Silv-Book}. Then we also estimate the size of the following set, 
\begin{align*}
&\cD_{\ell}(H)\\ 
&\quad= \{ (P_1, \ldots, P_s):P_i \in E_{i}(\Q),\ \h(P_i) \le H, \ i =1, \ldots, s,\ \text{and}\\
& \qquad\qquad\quad\quad d_{P_1}^{k_1} \cdots d_{P_s}^{k_s} \in \cN_{\ell},~\text{for some }(k_1,\ldots,k_s)\in [1,\ell-1]^s\}.
\end{align*}

\subsection{Main results}\label{sec:results}

\begin{thm}\label{thm:prod}
Let $s\ge 2$ be a fixed integer and $\ell$ be a fixed prime. Then uniformly over $M \ge 0$ and $\Qsf=(Q_1,\ldots, Q_s)\in E_1(\Q)\times \cdots E_s(\Q)$, we have
\[
\sharp\, \cD_{\ell,\Psf,\Qsf}(M,N) \;\ll\; \frac{N^s}{(\log N)^{\lceil s/2 \rceil/3}}.
\]
\end{thm}

Furthermore, we have a better estimate over the products of the terms of \textit{elliptic divisibility sequences}, that is, when $Q_i=O_i$ for each $i=1,\ldots,s$, where $O_i$ denotes the point at infinity for $E_i$.

\begin{thm}\label{thm:overEDS}
    Let $s\ge 2$ be a fixed integer and $\ell$ be a fixed prime. Then uniformly over $M \ge 0$, we have
\begin{equation}\label{eqn:generalM}
\sharp\, \cD_{\ell,\Psf,\isf}(M,N) \;\ll\; N^s\( \frac{\log \log N}{\log N}\)^s,
\end{equation}
where we denote $\isf=(O_1,\ldots, O_s)$, and for $M=0$ we have
\begin{equation}\label{eqn:M0}
\sharp\, \cD_{\ell,\Psf,\isf}(0,N) \;\le\; \frac{N^s}{\exp\(\( \frac{1}{\sqrt{2}} \lceil \frac{s}{2} \rceil +o(1)\)(\log N \log \log N)^{1/2}\)}.
\end{equation}
\end{thm}

Certainly, a more interesting question is whether one can obtain a power saving. Unfortunately, we do not currently have an affirmative answer. See Section~\ref{sec:abc} for some conditional perspectives, and the second paragraph of Section~\ref{sec:bottleneck} for further remarks, which provide a power saving over the integers in $(0,N]$ with at most a fixed number of prime factors, when counted with multiplicity.

Let us also arrange the rational points with respect to their canonical heights. Suppose $E(\Q)$ has rank $r$, then by~\cite[Theorem 4.5]{Kowalski}, we have
\begin{equation*}
H^{r/2} \ll  \sharp\, \{P \in E(\Q):~\h(P) \le H\} \ll H^{r/2}.
\end{equation*}

Using the same arguments as in the proof of Theorem~\ref{thm:prod}, we also have the following result.

\begin{thm}\label{thm:prodoverallpts}
Let $s\ge 2$ be a fixed integer and $\ell$ be a fixed prime and let $E_1, \ldots, E_s $ be elliptic curves over $\Q$ whose ranks over $\Q$ are $r_1,\ldots, r_s$, respectively. Assume that each $r_i>0$. Then uniformly over $H\ge 2$, we have 
$$\sharp\, \cD_{\ell}(H)\ll 
\frac{H^{(r_1+\cdots+r_s)/2}}{(\log H)^{\lceil s/2 \rceil/(1+2/r_{\min})}},$$
where we denote $r_{\min}=\min\{r_1,\ldots, r_s\}$.
\end{thm}

Furthermore, we obtain analogues of all these results for products of the $x$-coordinates (see Section~\ref{sec:CcapH}). In fact, denoting
\begin{align*}
X_{\ell,\Psf}(N)
=\sharp\, \left\{ (n_1, \ldots, n_s)\in (0,N]^s :\,
x(n_1P_1)^{k_1} \cdots x(n_sP_s)^{k_s} \in \cQ_{\ell}, \right.\\
\left. \qquad\qquad\qquad\qquad\qquad\qquad \text{for some } (k_1,\ldots,k_s)\in [1,\ell-1]^s \right\},
\end{align*}
we prove the following estimate.
\begin{thm}\label{thm:withx}
Let $s\ge 2$ be a fixed integer and $\ell$ be a fixed prime. Then we have
    \begin{equation*}
X_{\ell,\Psf}(N) \;\le\; \frac{N^s}{\exp\(\( \frac{1}{\sqrt{2}} \lceil \frac{s}{2} \rceil +o(1)\)(\log N \log \log N)^{1/2}\)}.
\end{equation*}
Moreover, assume that for each $i=1,\ldots,s$, we have 
\begin{equation}\label{eqn:mp}
x(m_iP_i)=0\quad \text{for some } m_i\in\mathbb N.
\end{equation}
Then the same upper bound also holds for $X_{2,P}(N)$.
\end{thm}

Note that \cite[Theorem 5.3]{BLMN} also imposes \eqref{eqn:mp}, but deriving an estimate for a significantly denser set. Thus, it is natural to ask the following question.
\begin{question}
    Does the same bound for $X_{2,P}(N)$ remain valid without the assumption \eqref{eqn:mp}?
\end{question}

Furthermore, the techniques used to prove these results also apply to sequences that are significantly denser than the set of $\ell$th powers, such as the integers or rationals representable as sums of two squares; see Section~\ref{sec:sumofsquares}.

\subsection{Further motivations}\label{sec:moremotivs}
Consider the toy model for estimating the number of pairs $(n_1,n_2)\in (M,M+N]^2$, for which $d_{n_1P}\cdot d_{n_2P}$ is a perfect square. Equivalently, this amounts to counting pairs $(n_1,n_2)\in (M,M+N]^2$ such that $u_{n_1}=u_{n_2}$, where $u_n$ denotes the squarefree part of $d_{nP}$, that is, the product of all distinct primes dividing $d_{nP}$ with an odd valuation. Following the treatment in \cite{BBHOS26}, it is useful to know a \textit{quadratic exponential} growth of $u_n$, that is, a Siegel-type property as in \cite[Lemma 3.1]{BBHOS26}. However, to the best of our knowledge, only the $abc$-conjecture implies such a growth of $u_n$; see the proof of \cite[Theorem 2]{Silv88} and some related discussions in Section~\ref{sec:abc}. In particular, the $abc$-conjecture implies that the number of pairs $(n_1,n_2)\in (M,M+N]^2$ with $u_{n_1}=u_{n_2}$ is of order $N$. Hence, by the results in Section~\ref{sec:results}, we obtain an unconditional nontrivial upper bound, which in turn also provides a lower bound for the number of distinct values $u_n$ with $n\in (M,M+N]$. See also Remark~\ref{rem:mostun} for some unconditional lower bounds on $u_n$.

The proper algebraic subgroups of $\mathbb{G}_m^s$ are defined by finitely many equations of the form $x_1^{k_1}\cdots x_s^{k_s} = 1$, where $k_1,\ldots,k_s \in \mathbb{Z}$ are not all zero. In \cite{BMZ99}, Bombieri, Masser, and Zannier have initiated the study of intersections of geometrically irreducible algebraic curves $\cC$ in $\mathbb{G}_m^s$ (for $s \ge 2$) defined over a number field $K$, with proper algebraic subgroups of $\mathbb{G}_m^s$. More precisely, if $\cC$ is not contained
in any translate of a proper algebraic subgroup of $\mathbb{G}_m^s$, then \cite{BMZ99} shows that the set of points in $\cC(\overline{\Q})$ with multiplicatively dependent coordinates, is a set of a bounded Weil height. Later, in \cite[Theorem 2.1]{OSSZ18}, the authors have proved finiteness of such points over the maximal abelian extension of $K$, when $\cC$ has positive genus. This shows finiteness of \textit{complex points} with multiplicatively dependent coordinates on elliptic curves with complex multiplication (CM). Later Barroero and Sha \cite{BarSha} have removed the CM condition.

Here we consider the product of elliptic curves $\mathbb{E}:=E_1 \times \cdots \times E_s$, and the map
\[
\Phi_X : \mathbb{E}^{\circ} \longrightarrow \mathbb{G}_m^{s}, \qquad
(P_1,\ldots,P_s)\mapsto \bigl(x(P_1),\ldots,x(P_s)\bigr),
\]
where $\mathbb{E}^{\circ}$ is the subset of $\mathbb{E}$ given by
\[
\mathbb{E}^{\circ} :=
\left\{(P_1,\ldots,P_s)\in \mathbb{E} :
P_i \neq O_i,\; x(P_i)\neq 0, \; \text{for } i=1,\ldots,s
\right\}.
\]

Recently, a different viewpoint has been considered in~\cite{BBHOS26}, where the authors study the intersections of $\Phi_{X}(\mathbb{E}^{\circ}(\Q))$ with proper algebraic subgroups of $\mathbb{G}_m^{s}$. In this paper, we instead consider the intersections of $\Phi_{X}(\mathbb{E}^{\circ}(\Q))\times \mathbb{G}_m$ with \textit{typical} (see Section~\ref{sec:admexponents}) proper algebraic subgroups of $\mathbb{G}_m^{s+1}$. Moreover, the arguments used to prove the results in Section~\ref{sec:results} are applied to study the analogous problems for the maps defined by non-constant rational functions $f_i \in \mathbb{Q}(E_i)$, $i=1,\ldots,s$, provided that at least one of the $f_i$ has either a zero or a pole at $O_i$. For further details, we refer the reader to Section~\ref{sec:CcapH}.

\subsection{Overview and the strategies of the proofs}\label{sec:strats}
In our setting, we do not always (e.g., in Theorem~\ref{thm:prod} and Theorem~\ref{thm:prodoverallpts}) restrict to the terms of an elliptic divisibility sequence. Therefore, the convenient $p$-adic properties from \cite[Lemma~2.1]{HLS} are no longer available. To address this, we borrow the general counting strategy from \cite{BHOS,BBHOS26}, suitably adapted to our context. The main distinction is that, here we need to count the denominators that are perfect $\ell$th powers up to a $\mathcal{W}$-unit, for some appropriately chosen finite set of primes $\mathcal{W}$. While in \cite{BBHOS26}, it sufficed to count (Siegel-type estimates) only the terms that are $\mathcal{W}$-units.

The counting is carried out in Lemma~\ref{lem:powerwithsunit}, using a factorization argument similar to the proof of Siegel’s lemma \cite[Theorem 4.3]{Silv-Book}, and closely following the argument in~\cite{ERS}. This problem is, in essence, concerned with estimating the number of rational points on a family of \textit{generalised Fermat curves} with some $\mathcal{W}$-unit coefficients. To estimate this, we use Lemma~\ref{lem:Gao21} of Dimitrov, Gao, and Habegger. Furthermore, Lemma~\ref{lem:powerwithsunit} allows us to remove the assumption \textit{Hypothesis~4.4} imposed in the recent preprint of Kym~\cite{Kym26}; see Section~\ref{sec:bottleneck} for more details.

Furthermore, in Section~\ref{sec:wunitoverEDS}, we also study these questions for products of terms of elliptic divisibility sequences via a version of the elliptic sieve due to Loughran, Myerson, Nakahara, and the present author \cite{BLMN}. This is essentially a sieving with respect to the reductions $E(\Q) \to E(\Z/p\Z)$, for which the order of the reduction of the base point $P$ modulo $p$ is a prime, in the sense of Kowalski~\cite[Section~4.4]{Kowalski} (it is also worth noting that a sieve for \textit{strong divisibility sequences} has recently appeared by Browning and Verzobio in~\cite{BV24}). 

In fact, as in \cite{BLMN}, we also need to sieve with respect to the reductions $E(\Q) \to E(\Z/p^{m\ell}\Z),~m \ge 1$, in analogy with the case of integers over intervals; see, for example, \cite[Section~4.8]{Tenenbaum}. This is also the same sieve set-up used in \cite{BLMN}, and we follow that in the proof of Lemma~\ref{lem:sq+sqwithsunit}. One of the key differences is that, since we consider powers up to a $\mathcal{W}$-unit, we need to restrict to primes $p$ exceeding a threshold (as in \eqref{eqn:threshold}) that depends on $\mathcal{W}$. 

However, our situation is considerably simpler than in \cite{BLMN}, where the primes $p$ are required to avoid certain arithmetic progressions, whereas no such restriction is present for us. This simplifies the argument, allowing an additional logarithmic saving. In fact, it turns out that over long intervals, the main contribution comes only from the smooth numbers (Lemma~\ref{lem:sq+sqwithsunitm0}), leading to a \textit{subpolynomial} saving in \eqref{eqn:M0}. In particular, both Lemma~\ref{lem:sq+sqwithsunit} and Lemma~\ref{lem:sq+sqwithsunitm0} (see Remark~\ref{rem:R>logN} for example) provide nontrivial estimates over some higher ranges of $\sharp\, \cW$, where Lemma~\ref{lem:powerwithsunit} holds trivially. As a consequence, we prove Theorem~\ref{thm:withx}, which gives an elliptic analogue of \cite[Theorem~2.2]{BKS21} for the $x$-coordinates.

We prove Theorem~\ref{thm:prodoverallpts} in Section~\ref{sec:prodoverallpts}. The key ingredient lies in Section~\ref{sec:maria}, where we estimate the number of points of bounded canonical height that reduce to $O$ modulo a large prime. This is obtained by counting lattice points in boxes over a family of sublattices of $\Z^r$ of large index. 

In Section~\ref{sec:applications}, we collect several applications of our methods. First we prove Theorem~\ref{thm:singlevar}, which serves as a version of \cite[Theorem~1]{HLS} for arbitrary shifts. Then in Section~\ref{sec:CcapH}, we discuss analogues of the results in Section~\ref{sec:results} for $x$-coordinates and more general rational functions $f \in \Q(E)$. Lastly in Section~\ref{sec:sumofsquares}, we present an application of the elliptic sieve in estimating the number of $s$-tuples of rational points, for which the product of their $y$-coordinates is a sum of two squares. Analogous to Lemma~\ref{lem:sq+sqwithsunit}, the main ingredient here is in \eqref{eqn:[]+[]uptoW}, which proves Theorem~\ref{thm:y=[]+[]} and recovers \cite[Theorem 1.1]{BLMN} for $s=1$.

Finally in Section~\ref{sec:bottleneck}, we address some challenges in obtaining power savings in our main results. Subsequently, in Section~\ref{sec:abc}, we outline approaches for obtaining power savings under variants of the $abc$-conjecture, indicating possible directions for further improvement.

\subsection{Notations}
We recall the standard convention that the notations $U \ll V$ and $U = O(V)$ are equivalent to $|U| \le cV$ for some constant $c>0$, where $c$ may depend on $s,\ell$, the $s$-tuples $\Psf$ and $\Qsf$ as in~\eqref{eq: P Q}, the curves $E_1,\ldots,E_s$ (i.e., the coefficients of their Weierstrass equations~\eqref{eq:Weier}), and, in some instances, their ranks. Furthermore we write $V=U^{o(1)}$ for any quantity $V$, which for any $\delta > 0$ satisfies $|V| \le c(\delta) U^{\delta}$, with some constant $c(\delta)$ that 
may depend on $\delta$ and all the invariants mentioned earlier. In particular, $U^{o(1)} \cdot U^{o(1)} = U^{o(1)}$. 

For any finite set $\mathcal{S}$, we denote by $\sharp\,  \mathcal{S}$ its cardinality. Given a prime $p$ and a nonzero rational number $x$, we denote by $\nu_p(x)$ the $p$-adic valuation of $x$. Finally, as usual, $\Q$, $\Z$, and $\N$ denote the sets of rational numbers, integers, and natural numbers, respectively.

\section{Perfect powers among denominators}
Let $K$ be a number field. We denote by $M_K$ the set of all places of $K$, both archimedean and non-archimedean, and let $M_{K,\infty}\subset M_K$ be the set of all archimedean places. 

Let $\mathcal{S} \subset M_K$ be a finite set of places containing $M_{K,\infty}$. The ring $\mathcal{O}_{K,\cS}$ of $\cS$-integers in $K$ is given by
$$\mathcal{O}_{K,\cS}=\{x \in K : \nu(x) \ge 0,~\mathrm{for~all}~\nu \in M_K\setminus \cS\},$$
and the unit group $\mathcal{O}^{*}_{K,\cS}$ of $\mathcal{O}_{K,\cS}$ is given by
$$\mathcal{O}^{*}_{K,\cS} = \{x \in K : \nu(x) = 0,~\mathrm{for~all}~\nu \in M_K\setminus \cS\}.$$

Let $C$ be a smooth, geometrically irreducible, projective curve of genus $g \ge 2$ defined over $\Q$, and $\J(C)$ be the Jacobian of $C$. We denote by $\rank\bigl(\J(C)(K)\bigr)$ the Mordell-Weil rank of $\J(C)$ over $K$. 

We shall use the following estimate due to Dimitrov, Gao and Habegger~\cite[Theorem~1.1]{DGH21}.

\begin{lem}\label{lem:Gao21}
There exists a constant $c(K,g) > 0$, such that
\[
\sharp\, C(K) \le c(K,g)^{1 + \rank\bigl(\J(C)(K)\bigr)}.
\]
\end{lem}

\subsection{A family of generalised Fermat curves}
Let $\mathcal{W}_K$ be a finite set of places of $K$ containing $M_{K,\infty}$ and at least one non-archimedean place, and let $\ell$ be as in Theorem~\ref{thm:prod}. Consider the following family of curves
\begin{equation}\label{eqn:Cuvell}
C_{u,v,\ell}: ux^{2\ell} + vy^{2\ell} = 1, 
\quad u,v \in \mathcal{O}^{*}_{K,\cW_K}/(\mathcal{O}^{*}_{K,\cW_K})^{2\ell}.
\end{equation}

We need to estimate the following sum
\begin{equation}\label{eqn:sumofsizes}
\sum_{u,v \in \mathcal{O}^{*}_{K,\cW_K}/(\mathcal{O}^{*}_{K,\cW_K})^{2\ell}}
\sharp\, C_{u,v,\ell}(K).
\end{equation}

Since $u,v \neq 0$, each such affine curve $C_{u,v,\ell}$ is smooth. Let $\widetilde{C}_{u,v,\ell}$ denote its projective model, given by $\widetilde{C}_{u,v,\ell}: ux^{2\ell} + vy^{2\ell} = z^{2\ell}$. This is a smooth, geometrically irreducible projective curve. It is well-known that its genus is $(2\ell-1)(\ell-1)$; see, e.g., \cite[Exercise 2.7]{Silv-Book}. Moreover, each of them is geometrically irreducible, since over $\overline{K}$ one can rescale $x$ and $y$ to reduce the defining equation to $x^{2\ell} + y^{2\ell}-z^{2\ell}=0$, which is irreducible over $\overline{K}$ (for instance, by Eisenstein’s criterion).

Since $C_{u,v,\ell}$ is an affine piece of $\widetilde{C}_{u,v,\ell}$, clearly we have $\sharp\, C_{u,v,\ell}(K) \le \sharp\, \widetilde{C}_{u,v,\ell}(K)$, and hence, we can apply Lemma~\ref{lem:Gao21} to $\widetilde{C}_{u,v,\ell}$.

To estimate \eqref{eqn:sumofsizes} we argue as in the proof of \cite[Proposition~10.1]{R10}, except that we use Lemma~\ref{lem:Gao21}, instead of \cite[Theorem 1.3]{R10} and the bound at the top of page 760 \cite{R10}.

\begin{lem}\label{lem:C'u}
We have 
  \[\sum_{u,v \in \mathcal{O}^{*}_{K,\cW_K}/(\mathcal{O}^{*}_{K,\cW_K})^{2\ell}}\sharp\,C_{u,v,\ell}(K)   \le \exp\Bigl(O\Bigl(L(\cW)\Bigr)\Bigr),
  \]
where the implied constant depends only on $\ell$ and $K$. Here, we denote
\begin{equation}\label{eqn:remond}
L(\cW)=\sum_{p\in\cW}\log p,
\end{equation}
and $\cW$ is the set of all rational primes lying below $\cW_K$.
\end{lem}

\begin{proof}
For any $u,v \in \mathcal{O}^{*}_{K,\cW_K}/(\mathcal{O}^{*}_{K,\cW_K})^{2\ell}$, denoting 
$$r(u,v)=\rank\bigl(\J(\widetilde{C}_{u,v,\ell})(K)\bigr),$$ Lemma~\ref{lem:Gao21} shows that
\[\sharp\, C_{u,v,\ell}(K)\le \sharp\, \widetilde{C}_{u,v,\ell}(K)
\le \exp\bigl(O(r(u,v)+1)\bigr).\]

On the other hand, the proof of \cite[Theorem 1]{OT89} or also \cite[Proposition 5.1]{R10}, shows that 
\[
r(u,v)\ll \log \Big|N_{K/\Q}\bigl(\cN^0_{\J(\widetilde{C}_{u,v,\ell})/K}\bigr)\Big|+1,
\]
where $\cN^{0}_{A/K}$ denotes the radical of the conductor of $A/K$, for any abelian variety $A$ over $K$. 

Note that for any curve $C$ over $K$, if $C$ has a good reduction at some prime $\mathfrak{p}$, then so does $\J(C)$ at $\mathfrak{p}$; see for instance \cite[Corollary 12.3]{Milne86}. In particular, $\widetilde{C}_{u,v,\ell}$ has a good reduction at any prime of $K$ except possibly those in $\mathcal{W}_K$ and those lying above $2$ or $\ell$. We derive
$$r(u,v)\ll L(\cW),~\mathrm{for~any}~u,v \in \mathcal{O}^{*}_{K,\cW_K}/(\mathcal{O}^{*}_{K,\cW_K})^{2\ell},$$
and the proof concludes.
\end{proof}

\subsection{Perfect powers up to a $\cW$-unit}
It is known, due to \cite{ERS}, that there are only $O(1)$ many $P\in E(\Q)$ such that $d_{P}$ is a perfect $\ell$th power. We need to estimate the following slightly modified quantity, which will be needed to prove Theorem~\ref{thm:prod}. 

For any prime $\ell$, and any set of rational primes $\cW$, we denote
\begin{equation}\label{eqn:dlw}
\cD_{\ell,\cW}=\{P\in E(\Q):d_{P}\in \mathcal{O}^{*}_{\Q,\cW} \cdot \cN_{\ell}\}.
\end{equation}

We now prove the following estimate. The proof of Lemma~\ref{lem:powerwithsunit} below, follows almost verbatim the argument in \cite[Theorem~1.1]{ERS}. For completeness, we include a brief self-contained proof.

\begin{lem}\label{lem:powerwithsunit}
    Let $\ell$ be any fixed prime, and $\cW$ be any finite non-empty set of rational primes. Then we have 
\begin{align*}
\sharp\,\cD_{\ell,\cW}\le \exp\Bigl(O\Bigl(L(\cW)\Bigr)\Bigr),
\end{align*}
where $L(\cW)$ be as in \eqref{eqn:remond}.
\end{lem}

\begin{proof}
 Let $\alpha_1,\alpha_2,\alpha_3$ be the (distinct) roots of $x^3+ax+b=0$. So for any $P\in E(\Q)$, we have
 \begin{equation}\label{eqn:maineqn}
 b^2_P=(a_P-\alpha_1 d^2_P)(a_P-\alpha_2 d^2_P)(a_P-\alpha_3 d^2_P).
 \end{equation}
 
To facilitate the factorisation argument, let $K = \Q(\alpha_1,\alpha_2,\alpha_3)$, and choose a finite set of places $\cS \subset M_K$ containing $M_{K,\infty}$ such that $\cO_{K,\cS}$ is a PID, and $2,\alpha_i-\alpha_j \in \cO_{K,\cS}^\ast$ for any $i \ne j$. 

Then, we denote $L$ be the finite extension of $K$, generated by the square-roots of the $\cO_{K,\cS}^{*}$. Let $\cT \subset M_{L}$ be any finite set of places lying over the elements of $S$ such that $\cO_{L,\cT}$ is a PID. 

Now factorising \eqref{eqn:maineqn}, we can write 
\begin{equation}\label{eqn:abetaz}
a_P-\alpha_i d_P^2 = (\beta_iz_i)^2,~\mathrm{for}~i=1,2,3,
\end{equation}
for some $\beta_i,z_i\in \mathcal{O}_{L,\cT}$. In particular, we get the following equations
\begin{equation}\label{eqn:differences}
    (\alpha_i-\alpha_j)d^2_P=(\beta_iz_i-\beta_jz_j)(\beta_iz_i+\beta_jz_j),\quad i\ne j \in \{1,2,3\},
\end{equation}
with $\alpha_i-\alpha_j\in \mathcal{O}^{*}_{L,\cT}$ and $\beta_i,z_i\in \mathcal{O}_{L,\cT}$, for each $i=1,2,3$.

Now, we need to make a modification to the argument in \cite[Theorem 1.1]{ERS}, as our $d_P$ is a perfect $\ell$th power, but up to some $\cW$-unit. For this,  we consider the following set of primes
$$\cT_{\cW} = \cT \cup \{ \mathfrak{p} \in M_{L} : \mathfrak{p} \mid p,\; p \in \cW \}.$$

As we have $(a_P,d_P)=1$, it follows from \eqref{eqn:abetaz} that the factors on the right side of \eqref{eqn:differences} are coprime in $\mathcal{O}_{L,\cT_{\cW}}$. In particular, we can write
$$ \beta_iz_i \pm \beta_jz_j \in \mathcal{O}^{*}_{L,\cT_{\cW}} \cdot \(\mathcal{O}_{L,\cT_{\cW}}\)^{2\ell},\quad \mathrm{for~any}~ i, j \in \{1,2,3\}.$$
Then using Siegel's identity 
\[
\frac{\beta_1 z_1 \pm \beta_2 z_2}{\beta_1 z_1 - \beta_3 z_3} 
\;\; \mp \;\;
\frac{\beta_2 z_2 \pm \beta_3 z_3}{\beta_1 z_1 - \beta_3 z_3} 
= 1,
\]
from \eqref{eqn:differences} we get a system of equations of the form
\begin{equation}\label{eqn:allequations}
ux^{2\ell} + vy^{2\ell} = 1,\quad u,v\in \mathcal{O}^{*}_{L,\cT_{\cW}}/(\mathcal{O}^{*}_{L,\cT_{\cW}})^{2\ell},\quad x,y \in L.
 \end{equation}

Following the concluding argument in the proof of~\cite[Theorem~1, p.~767]{ERS}, note that each such solution $(x,y)$ in \eqref{eqn:allequations} contributes only $O(1)$ rational points in $\cD_{\ell,\cW}$. The proof concludes, applying Lemma~\ref{lem:C'u}.
\end{proof}

\begin{rem}Alternatively, one can also write
    $$\sharp\,\cD_{\ell,\cW}\le \sum_{u\in \mathcal{O}^{*}_{\Q,\cW}/\(\mathcal{O}^{*}_{\Q,\cW}\)^{\ell}}\sharp\,\cD_{\ell}(u),$$
    where we denote
$$\cD_{\ell}(u) = \{ P \in E_u(\Q) : d_P \in \cN_{\ell}\}\quad \mathrm{and}\quad E_u: y^2 = x^3 + a u^4 x + b u^6.
$$

However, as is evident from the proof of Lemma~\ref{lem:powerwithsunit} or \cite[Theorem~1.1]{ERS}, the resulting bound depends explicitly on the coefficients of the Weierstrass equation \eqref{eq:Weier}. In particular, with the family of elliptic curves $E_u$ where the coefficients vary with $u$, one is essentially led back to the same family of generalised Fermat curves as in \eqref{eqn:allequations}.
\end{rem}

\section{Products over shifted points}

\subsection{Preliminary}
Let us first recall the following ingredients from \cite{BBHOS26}, which will be used throughout.

\begin{definition}[Index of appearance]
For a prime $p$, we denote by $\rho_p$ the index of appearance of $p$ as a divisor in the sequence $d_{nP}$, $n=1, 2, \ldots$, that is, the smallest $r$ such that $d_{rP} \equiv 0 \pmod p$; we set $\rho_p = \infty$ if no such $r$ exists, with the natural rules of operating with this quantity (as, for example, $\infty^{-1} = 0$). 
\end{definition}

We now estimate the number of terms with a large prime factor. The following lemma is taken from \cite[Lemma 3.3]{BBHOS26}.

\begin{lem}\label{lem:divbyp} 
    For any prime $p$, we have
\[\sharp\,\{M<n\leq M+N :~\nu_p\(x(nP+Q)\) \ne 0 \}\ll \frac{N}{\rho_p}+ 1.
\]
\end{lem}

For any positive real number $R$, let $ \cW(R)$ be
the set of all primes $p$ with $\rho_{i,p} \le R$ for at least one $i =1, \ldots, s$. Here $\rho_{i,p}$ denotes the index of appearance of $p$ with respect to $P_i$, for each $i=1,\ldots, s$, where $P_1,\ldots,P_s$ are the rational points as in Theorem~\ref{thm:prod}.

\begin{lem}\label{lem:LandW}
    For any $R>1$, we have
    $$L(\cW(R))\ll R^3 \quad \mathrm{and}\quad \sharp\, \cW(R) \ll R^3/\log R.$$
\end{lem}

\begin{proof}
Note that if $\rho_{i,p}\le R$ for some $1\le i\le s$, then $p\mid d_{rP_i}$ for some integer $r$ with $1\le r\le R$. In particular, we can write
    $$L(\cW(R))\le \sum_{\substack{1\le r\le R\\ 1\le i\le s}} \log d_{rP_i}.$$
    
Hence, the estimation of $L(\cW(R))$ follows from \cite[Lemma 3.1]{BBHOS26}. On the other hand, the estimation of $\sharp\, \cW(R)$ is already proved in \cite[Lemma 3.4]{BBHOS26}.
\end{proof}

We shall also appeal to the following lemma from~\cite[Lemma~2.7]{BHOS}, which provides a saving that grows with~$s$.

\begin{lem}\label{lem:GraphCover}
Assume we are given a graph with the vertex set $\cV$ of cardinality  $t=\sharp\,  \cV$ and having no isolated vertex. Then there exists $ \cV_1\subseteq  \cV$ with $\sharp\,  \cV_1\leq t/2$ such that for any $v_2\in  \cV_2= \cV\setminus  \cV_1$ there exists a vertex $v_1\in  \cV_1$ which is a neighbour of $v_2$.
\end{lem}

\subsection{Proof of Theorem~\ref{thm:prod}}
First, by Siegel's theorem, except for at most $O(N^{s-1})$ many tuples we may assume that $d_{n_iP_i+Q_i}>1$, for each $i=1,\ldots, s$.

Now, let $1<R< N$ be a real number to be chosen suitably later. We first count the $s$-tuples of \textit{first type}, that is, the $s$-tuples $(n_1, \ldots, n_s)\in \cD_{\ell,\Psf,\Qsf}(M,N)$, for which 
$$\nu_p(d_{n_iP_i+Q_i})\not\equiv 0 \pmod \ell \implies p \in \cW(R),~\mathrm{for~some }~i=1,\ldots, s.$$

Lemma~\ref{lem:powerwithsunit} and Lemma~\ref{lem:LandW} show that, the number of such $s$-tuples of the first type is bounded by $$ N^{s-1} \exp\left(O\(R^3\)\right).$$

For each remaining $s$-tuple $(n_1,\ldots, n_s)\in \cD_{\ell,\Psf,\Qsf}(M,N)$ of \textit{second type}, and for any $i=1,\ldots, s$, there exists a prime $p\not\in \cW(R)$ with $$\nu_p(d_{n_iP_i+Q_i})\not\equiv 0 \pmod \ell.$$

In particular, for any such tuple $(n_1,\ldots, n_s)$ of second type, and for any $i=1,\ldots,s$, there exists some $1 \le j\le s$ with $j\ne i$, satisfying
\begin{equation}\label{eqn:edge}
\nu_p(d_{n_iP_i+Q_i})\not\equiv 0 \pmod \ell \quad \mathrm{and}\quad \nu_p(d_{n_jP_j+Q_j})\not\equiv 0 \pmod \ell,
\end{equation}
for some prime $p\not\in \cW(R)$. 

We claim that the number of such $s$-tuples of second type is 
\begin{equation}\label{eqn:largeones}
O\(N^{s}R^{-\rf{s/2}}\).
\end{equation}

To prove this claim, we consider the graph $\cG$ on $s$ vertices  $n_1, \ldots, n_s$, and the edges are defined by the condition at \eqref{eqn:edge}. By the observation above, $\cG$ has no isolated vertex. By Lemma~\ref{lem:GraphCover}, we can fix a
subset $\cI$ of $\{1,\ldots, s\}$ with $m=\sharp\,  \cI\leq \fl{s/2}$, such that for any $j\in \{1,\ldots,s\}\setminus \cI$, there exists some $i\in \cI$, such that $n_i$ and $n_j$ satisfy \eqref{eqn:edge}.

There are trivially $O(N^m)$ choices for such $m$-tuples. Without loss of generality, we may assume that the fixed ones are $(n_1,\ldots, n_m)$. For any such fixed $m$-tuple, applying Lemma~\ref{lem:divbyp}, the number of remaining $(s-m)$-tuples $(n_{m+1},\ldots, n_s)\in (M,M+N]^{s-m}$ (such that $(n_1,\ldots,n_{m},n_{m+1},\ldots,n_s)$ is of second type) is
$$O\((N/R+1)^{s-m}\)= O\(N^{s-m}R^{-\rf{s/2}}\),$$ 
since $R< N$ and $m\le \lceil s/2 \rceil$. This proves the claim at \eqref{eqn:largeones}. For a more detailed argument, we also refer the reader to the estimations of $K_2$ and $K_3$ in the proof of \cite[Theorem 2.1]{BBHOS26}. 

Therefore, we finally derive
$$\sharp\, \cD_{\ell,\Psf,\Qsf}(M,N)\le N^{s-1} \exp\left(O\(R^3\)\right)+O\(N^{s}R^{-\rf{s/2}}\).$$

The proof concludes by taking
$R=c(\log N)^{1/3}$, for some sufficiently small constant $c>0$.
\qed

\section{Products over terms of elliptic divisibility sequences}

Let $\cW(R)$ be the same set of primes as in the proof of Theorem~\ref{thm:prod}. For any fixed $\ell$, we need to estimate the size of the following set 
$$\cD_{\ell,P,\cW(R)}(M,N)=\{M<n\le M+N: d_{nP}\in \mathcal{O}^{*}_{\Q,\cW(R)}\cdot \cN_{\ell}\}.$$

Let $\cW_P = \{p \ \text{prime} : p \mid d_P\}$, including $2$ if necessary. Clearly we have $L(\cW_P)\ll \log d_P$. Therefore Lemma~\ref{lem:powerwithsunit} shows that there exists a constant $C_{E,P,\ell}>0$, such that for any prime $q>C_{E,P,\ell}$ we have $qP\not \in \cD_{\ell,\cW_P}$, where $\cD_{\ell,\cW_P}$ be as in \eqref{eqn:dlw}.

\subsection{Elliptic sieve and powers up to a $\cW(R)$-unit}\label{sec:wunitoverEDS}
In this section we follow the sieving argument as in the proof of \cite[Theorem 1.6]{BLMN}, and prove stronger versions of Lemma~\ref{lem:powerwithsunit}, over the intervals whose lengths are much larger than $\sharp\, \cW(R)$.

\begin{lem}\label{lem:sq+sqwithsunit}
    Let $E/\Q$ be any elliptic curve as in $\eqref{eq:Weier}$ and $P\in E(\Q)$ be any non-torsion point. Then for any fixed $\ell$, we have
\begin{equation}\label{eqn:dlpw}
\sharp\, \cD_{\ell, P,\cW(R)}(M,N)\ll   N \, \frac{\log \, \max\{R,\log N\}}{\log  N}.
\end{equation}
\end{lem}

\begin{proof}
We can assume that $R\ge \log N$, as $\sharp\, \cD_{\ell, P,\cW(R)}(M,N)$ is an increasing function of $R$. 

Lemma~\ref{lem:powerwithsunit} implies that, for any prime $q>C_{E,P,\ell}$, there exists a prime $p_q \not\in \cW_P$ with 
$$\nu_{p_q}(d_{qP})\ne 0\pmod {\ell}.$$
We then consider the following set of primes
$$
\Lambda_P(R)=\left\{q \text{ prime} : q>R+C_{E,P,\ell}\right\}.
$$
For any $q \in \Lambda_P(R)$, we have $p_q \notin \cW(R)$. To see this, note that $p_q \mid d_{qP}$. Since $p_q$ does not divide $d_P$, we have 
\begin{equation}\label{eqn:threshold}
\rho_{p_q} = q > R,
\end{equation}
and hence $p_q \not \in \cW(R)$, as we claimed.

Let $n\in (M,M+N]$ and $q\in \Lambda_{P}(R)$ with $q\mid n$ and $\gcd(p_q,n/q)=1$. 
Then, by \cite[Lemmas 2.3, 2.4, and 3.4]{BLMN}, we have
\begin{equation}\label{eqn:idenity}
\nu_{p_q}(d_{nP})=\nu_{p_q}(d_{qP})\ne 0 \pmod {\ell}.
\end{equation}
In particular, any such $n$ is not in $\cD_{\ell,P,\cW(R)}(M,N)$. 

Without loss of generality we may assume that $R>C_{E,P,\ell}$, as otherwise Lemma~\ref{lem:sq+sqwithsunit} holds trivially. In particular, $\Lambda_P(R)$ contains the set of all primes larger than $2R$. 

Now, let $y$ be a parameter to be chosen suitably later, satisfying $2R<y\le M+N$. Denote $\cA_0(M,N; y)$ be the set of integers $n\in (M, M+N]$ without a prime factor in $(2R,y]$. Moreover, we denote $\cA_1(M,N; y)$ be the complement of $\cA_0(M,N; y)$ in $\cD_{\ell,P,\cW(R)}(M,N)$.

With the observation in \eqref{eqn:idenity}, we can write
$$\sharp\,\cD_{\ell,P,\cW(R)}(M,N)\le \sharp\, \cA_0(M,N; y) +\sharp \, \cA_1(M,N; y).$$

By the fundamental lemma of combinatorial sieve \cite[Theorem 4.4]{Tenenbaum},
\begin{align*}
\sharp\, \cA_0(M,N; y) &\ll N \prod_{2R<q \le y}\(1-\frac{1}{q}\)+y\ll N \, \frac{\log R}{\log y}+y.
\end{align*}
On the other hand, we also have 
\begin{align*}
			\sharp\, \cA_1(M,N; y)
			&\ll  \sum_{\substack{2R < q \le y}}
			\sharp\,\{ n\in (M,M+N] : q p_q \mid n \}\\ 
			&\ll  \sum_{\substack{2R < q \le y}}
			\(\frac{N}{q p_q}+1\)\ll  \sum_{2R<q\le y}\(\frac{N}{q^2}+1\)
		\ll \frac{N}{R}+y,
\end{align*}
where we use the fact that $q=\rho_{p_q}\ll p_q$, by Hasse-Weil bound.

On the other hand we can assume that $2R<N^{1/2}$, as otherwise \eqref{eqn:dlpw} holds trivially. The proof now concludes, by choosing $y=N^{1/2}$.   
\end{proof}

\begin{rem}\label{rem:R>logN}
Certainly, the estimate in Lemma~\ref{lem:sq+sqwithsunit} is nontrivial for any $R$ and $N$ satisfying $R=\exp\(\frac{\log N}{F(N)}\)$, with any function $F(n)\to \infty$ as $n\to \infty$. However, Lemma~\ref{lem:powerwithsunit} provides a nontrivial estimate only in the range $R\ll (\log N)^{1/3}$.
\end{rem}

Let us denote $\Psi(N,y)$ be the number of $y$-smooth positive integers up to $N$, and as usual, we say that a positive integer $n$ is $y$-smooth if any prime factor of $n$ is at most $y$. Now, we also have the following version of Lemma~\ref{lem:sq+sqwithsunit} for $M=0$.
\begin{lem}\label{lem:sq+sqwithsunitm0}
    Let $E/\Q$ be any elliptic curve as in $\eqref{eq:Weier}$ and $P\in E(\Q)$ be any non-torsion point. Then for any fixed $\ell$, we have
\begin{equation*}
\sharp\, \cD_{\ell, P,\cW(R)}(0,N)\ll  \, \Psi(N,2R)+N/R.
\end{equation*}
\end{lem}

\begin{proof}
We argue exactly as in the proof of Lemma~\ref{lem:sq+sqwithsunit}. The only difference is that, we take $y=N$. Then we have 
 $$\sharp\, \cA_0(0,N; N)\le \, \Psi(N,2R)\quad \mathrm{and}\quad \sharp\, \cA_1(0,N; N)\ll \, N/R,$$
as we can write
\begin{align*}
			\sharp\, \cA_1(0,N; N)
			&\ll  \sum_{\substack{2R < q \le N}}
			\sharp\,\{ n\in (0,N] : q p_q \mid n \}\\ 
			&\ll  \sum_{\substack{2R < q \le N}}
			\frac{N}{q p_q}\ll  \sum_{2R<q\le N}\frac{N}{q^2}
		\ll \frac{N}{R}.
\end{align*}
The proof concludes.
\end{proof}

Therefore, Lemma~\ref{lem:sq+sqwithsunitm0} suggests that we need a comparison between $\Psi(N,R)$ and $N/R$. Writing $u=\frac{\log N}{\log R}$, we deduce from \cite[Corollary, Page~15]{CEP83} that
\begin{equation}\label{eqn:smooth}
\Psi(N,R) \le N \exp\big((-1+o(1))\, u \log u\big),
\end{equation}
provided that $u  \to \infty$ and $R > (\log N)^2$. Therefore, taking any $R$ satisfying 
\begin{equation}\label{eqn:newR}
\frac{\log N}{\log R}\(\log \log N -\log \log R\)=\log R,
\end{equation}
we obtain 
\begin{equation}\label{eqn:psinrbound}
\Psi(N,R)\le N/R^{1+o(1)}.
\end{equation}

\subsection{Proof of Theorem~\ref{thm:overEDS}}
To prove \eqref{eqn:generalM}, we follow a similar argument as in the proof of Theorem~\ref{thm:prod}. 
The only difference is that we split the $s$-tuples $(n_1,\ldots,n_s)\in \cD_{\ell,\Psf,\isf}(M,N)$ into two types. 

First, those for which, for every $i=1,\ldots,s$ we have
\begin{equation}\label{eqn:foreach}
\nu_p(d_{n_iP_i}) \not\equiv 0 \pmod \ell \implies p \in \cW(R).
\end{equation}
Second, those for which \eqref{eqn:foreach} fails for at least one $1 \le i \le s$.

For any parameter $R$ with $\log N \le R < N$, Lemma~\ref{lem:sq+sqwithsunit} shows that the number of $s$-tuples satisfying \eqref{eqn:foreach} is $\ll N^s \left(\frac{\log R}{\log N}\right)^s$, while Lemma~\ref{lem:divbyp} implies that the number of $s$-tuples failing \eqref{eqn:foreach} is $\ll N^s / R$. 

Combining these bounds, we derive
\[
\sharp\, \cD_{\ell,\Psf,\isf}(M,N) \ll N^s \left(\frac{\log R}{\log N}\right)^s + \frac{N^s}{R}.
\]
The proof of \eqref{eqn:generalM} concludes by taking $R = (\log N)^s$.

To prove \eqref{eqn:M0}, we take the $R$ satisfying \eqref{eqn:newR}. Then for each $1\le t\le s$, we count the number of $s$-tuples $(n_1, \ldots, n_s) \in \cD_{\ell,\Psf,\isf}(0,N)$ for which \eqref{eqn:foreach} fails for $t$ many $i$ in $\{1,\ldots,s\}$. Without loss of generality, we may assume that the corresponding integers are  $n_1,\ldots,n_t$, and $n_{t+1},\ldots,n_s$, respectively.

By \eqref{eqn:psinrbound}, the number of such $(s-t)$-tuples $\(n_{t+1},\ldots,n_s\) \in (0,N]^{s-t}$ is bounded by
$$ \(\frac{N}{R^{1+o(1)}}\)^{s-t}.$$

As in the proof of Theorem~\ref{thm:prod}, we again consider the graph $\cG_{t}$, where the vertices are $(n_1,\ldots,n_t)$, and the edges are defined by the condition  
\begin{equation*}
\nu_p(d_{n_iP_i})\not\equiv 0 \pmod \ell \quad \mathrm{and}\quad \nu_p(d_{n_jP_j})\not\equiv 0 \pmod \ell,
\end{equation*}
for some prime $p\not \in \cW(R)$. Then, again as in the proof of Theorem~\ref{thm:prod}, applying Lemma~\ref{lem:divbyp} and Lemma~\ref{lem:GraphCover}, the number of such $t$-tuples $(n_1,\ldots, n_t)\in (0,N]^t$ is $O\(\frac{N^t}{R^{\rf{t/2}}}\).$

Therefore, summing over $t=0, \ldots, s$, we finally derive
\[
\sharp\, \cD_{\ell,\Psf,\isf}(0,N) \le \, \sum_{t=0}^{s}\(\frac{N}{R^{1+o(1)}}\)^{s-t} \, \frac{N^t}{R^{\rf{t/2}}}\le \, \frac{N^{s}}{R^{\lceil s/2 \rceil \(1+o(1)\)}}.
\]

Also, it is clear from \eqref{eqn:newR} that
\[(\log N \log \log N)^{1/2} \ge \log R \ge \frac{1}{\sqrt{2}}\left(\log N \, (\log\log N - 2\log\log\log N)\right)^{1/2},
\]
and the proof concludes.
\qed
\begin{rem}\label{rem:mostun}
   Recall the sequence $(u_n)$ introduced in Section~\ref{sec:moremotivs}. By the Hasse-Weil bound, for any $R>1$, we have
\[
\{n \in (0,N] : u_n < R\} \subseteq \cD_{2, P,\cW(O(R))}(0,N).
\]

In particular taking $R=\exp\(\frac{\log N}{F(N)}\)$, for any function $F(n)\to \infty$ as $n\to \infty$, it follows from \eqref{eqn:smooth} that $u_n$ admits a subpolynomial lower bound for almost all $n \in (0,N]$. This, however, is still far from the quadratic-exponential growth predicted by the abc-conjecture.
\end{rem}

\subsection{Proof of Theorem~\ref{thm:withx}}
We argue exactly as in the proof of Theorem~\ref{thm:overEDS}. For this, one of the ingredients is already available, due to Lemma~\ref{lem:divbyp}. Therefore, we only need an analogue of Lemma~\ref{lem:sq+sqwithsunitm0}. This is straightforward for $\ell>2$, as 
\begin{equation*}
x(nP)\in \mathcal{O}^{*}_{\Q,\cW(R)} \cdot \cQ_{\ell} \implies d_{nP}\in \mathcal{O}^{*}_{\Q,\cW(R)} \cdot \cN_{\ell}.
\end{equation*}

For $\ell=2$, we use the identity
\begin{equation*}
|x(nP) - x(mP)|=\frac{|\psi_{n+m}(P)\psi_{n-m}(P)|}{\psi^2_n(P)\psi^2_m(P)},\quad \mathrm{for~any }~n>m\ge 1,
\end{equation*}
where $\psi_n$ denotes the $n$th division polynomial; see \cite[Exercise 3.7]{Silv-Book}. Moreover, it follows from  \cite[Theorem A]{Ayad92} that $\nu_p(d_{nP})=\nu_p(|\psi_n(P)|)$ for any $n\ge 1$, and for any prime $p \nmid d_P$ at which $E$ has a good reduction.

In particular if $x(mP)=0$ for some $m\ge 1$, then for any $n>m$, we have
$$x(nP)\in \mathcal{O}^{*}_{\Q,\cW(R)} \cdot \cQ_{2} \implies d_{(n+m)P}\in \mathcal{O}^{*}_{\Q,\cW'(R)} \cdot \cN_{2},$$
where $\mathcal{W}'(R)$ is the union of $\mathcal{W}(R)$ with a finite set of primes depending only on $E$ and $P$. Thus, one can again use Lemma~\ref{lem:sq+sqwithsunitm0}. The proof concludes.
\qed

\section{Products over all points of a bounded height}\label{sec:prodoverallpts}
\subsection{Counting points in a family of lattices}\label{sec:maria}
Suppose that $E(\Q)$ has rank $r$, and $G_1, \dots, G_r$ generate the free part of $E(\Q)$. For any  prime $p$ of good reduction, consider the lattices in $\Z^r$ given by
\[
\Lambda_p = \left\{\vec{n} \in \Z^r : n_1 G_1+\cdots +n_rG_r \equiv O \pmod p\right\},
\]
where again $O$ is the point at infinity for $E$, and 
\[
\Lambda'_p = \left\{\vec{n} \in \Z^r : n_1 G_1+\cdots +n_rG_r \in E(\Q)_{\mathrm{tor}} \pmod p\right\}.
\]

Here, we denote $E(\Q)_{\mathrm{tor}}$ be the subgroup of torsion points in $E(\Q)$. Moreover, for the definition of reduction modulo prime $p$ of a rational point $P$, or equivalently, for the meaning of the standard notation $P \pmod p$, we refer the reader to \cite[Definition~2.2]{BLMN}.

\begin{lem}\label{lem:latticesoffullrank}
For any  prime $p$ of good reduction, both $\Lambda_p$ and $\Lambda'_p$ are lattices of full ranks in $\Z^r$.
\end{lem}

\begin{proof}
Denoting 
\begin{equation*}
\rho_{E,p}=\sharp\,\left\{ n_1 G_1 +\cdots+n_r G_r \pmod p\right\},
\end{equation*}
we have
\begin{equation}\label{eqn:LambdavsLambda'}
 \rho_{E,p} = \sharp\,\(\Z^r/\Lambda_p\)\le \sharp\,\{P \pmod p: P\in E(\Q)_{\mathrm{tor}}\}\cdot \sharp\,\(\Z^r/\Lambda'_p\),
\end{equation}
and
\begin{equation}\label{eqn:ep}
\sharp\,\(\Z^r/\Lambda'_p\)\le \rho_{E,p}.
\end{equation}

For any prime $p$ of good reduction, we have $\rho_{E,p}\le \sharp\, \widetilde{E}(\F_p)$, where $\widetilde{E}$ denotes the reduction of $E$ modulo $p$. In particular, by \eqref{eqn:ep} we find that both $\sharp\,\(\Z^r/\Lambda_p\)$ and $\sharp\,\(\Z^r/\Lambda'_p\)$ are finite, for each prime $p$ of good reduction. The proof concludes.
\end{proof}

Now consider the box
$$\cB(X)=\{\vec{x}\in \R^r: \|\vec{x}\|_{\infty}\le X\},$$
where $\|\cdot\|_{\infty}$ denotes the sup norm, and let $\Lambda$ be a lattice in $\Z^r$.

For any $\vec{m} \in \Z^r$, we denote by $N_{m,\Lambda}(X)$ the size of $\cB(X)\cap \(\vec{m}+\Lambda\)$, where by $\vec{m}+\Lambda$ we mean the translation of the lattice $\Lambda$ by the vector $\vec{m}$. Then, we record the following presumably well-known estimate.

\begin{lem}\label{lem:latticepointcounting}
Let $\Lambda$ be a lattice in $\Z^r$ of rank $r$. Then, we have
\[N_{m,\Lambda}(X)= \frac{\mathrm{Vol}\(\cB(X)\)}{\sharp\,(\Z^r/\Lambda)}+O\((X+\|\vec{m}\|_{\infty})^{r-1}\),
\]
where $\mathrm{Vol}(\mathcal{B}(X))$ is the volume of $\mathcal{B}(X)$, and the implied constant depends only on $r$.
\end{lem}

\begin{proof}
Consider the convex compact set $\cS=\cB(X)-\vec{m}$, contained in $\cB(X+\|\vec{m}\|_{\infty})$. Clearly, we have $N_{m,\Lambda}(X)=\sharp\, (\cS\cap \Lambda)$.

Denote by $\det(\Lambda)$ the determinant (or co-volume) of the lattice $\Lambda$, and let $\lambda_1\le \ldots \le \lambda_r$ be the successive minima of $\Lambda$. Following the proof of \cite[Lemma 1]{Sch95}, or in other words applying \cite{D51}, we have 
\begin{align*}
&\left|N_{m,\Lambda}(X)-\frac{\mathrm{Vol}\(\cB(X)\)}{\det(\Lambda)}\right|\\
&\qquad\qquad\quad\ll \(X+\|\vec{m}\|_{\infty}\)^{r-1}\max_{0\le i\le r-1}\frac{\lambda_{i+1} \ldots \lambda_{r}}{\det \Lambda}\\
&\qquad\qquad\quad \ll 1+ \(X+\|\vec{m}\|_{\infty}\)^{r-1}\max_{1\le i\le r-1}\frac{1}{\lambda_1\ldots \lambda_i},
\end{align*}
where we are using \cite[Equation~(2.1)]{Sch95} in the last line above, and the implied constants all depend on $r$. We get the desired error term, as clearly $\lambda_1\ge 1$, since $\Lambda$ is a lattice in $\Z^r$.

Since $\Lambda$ is of full rank, we can write 
$$\Lambda=\{A\cdot \vec{n}:\vec{n}\in \Z^r\},$$
for some invertible $r\times r$ matrix $A$ with integer entries. We then have, $\det(\Lambda)=|\det(A)|$. The proof concludes, noting that $|\det(A)|=\sharp\,(\Z^r/\Lambda)$; to see this, one can for instance use the Smith normal form (see \cite[Corollary 1.13]{N12}) of the matrix $A$.
\end{proof}

Let us now denote
$$\cE_{p}(H)=\{P\in E(\Q) :\h(P)\le H,~\nu_p\(x(P)\) \ne 0 \}.$$

As a consequence of Lemma~\ref{lem:latticepointcounting}, we deduce the following estimate.

\begin{lem}\label{lem:divbypoverallpointsbetter}
     Let $p$ be any prime of good reduction. We have
\begin{align*}
\sharp\, \cE_{p}(H)\ll \frac{H^{r/2}}{\rho_{E,p}}+H^{(r-1)/2}.
\end{align*}
\end{lem}

\begin{proof}
Note that 
$$P=n_1G_1+\cdots+n_rG_r+t,~t\in E(\Q)_{\mathrm{tor}}\quad \mathrm{with}\quad \nu_p(x(P))<0,$$
implies that 
$$\vec{n}=(n_1,\ldots,n_r)\in \Lambda'_p.$$

On the other hand, for any two points $P_1$ and $P_2$ with $$\nu_p(x(P_1))>0\quad \mathrm{and}\quad \nu_p(x(P_2))>0,$$ 
we have $P_1=\pm P_2 \pmod p$. In particular, either $\cE_{p}(H)$ is empty or there exists some $P_0\in \cE_{p}(H)$ with the following property: for any
$$P=n_1G_1+\cdots+n_rG_r+t,~t\in E(\Q)_{\mathrm{tor}}\quad \mathrm{with}\quad \nu_p(x(P))>0,$$
we have
$$\vec{n}=(n_1,\ldots,n_r)\in \pm \vec{m}+\Lambda'_p,$$
where we write
$$P_0=\sum_{i=1}^{r}m_iG_i+t_0,~t_0\in E(\Q)_{\mathrm{tor}},\quad \mathrm{and}\quad \vec{m}=(m_1,\ldots,m_r).$$

Since $P_0\in \cE_{p}(H)$, it follows that $\widehat{h}(P_0)\le H$. In particular, we have $\|\vec{m}\|_{\infty}\ll H^{1/2}$. This is because, $\|\vec{m}\|^2_{\infty}\ll \widehat{h}(P_0)$; see the proof of~\cite[Theorem~4.5]{Kowalski} for this well-known inequality. The proof now concludes by Lemma~\ref{lem:latticepointcounting} and Lemma~\ref{lem:latticesoffullrank}, as \eqref{eqn:LambdavsLambda'} implies $\rho_{E,p}\ll \sharp\,\(\Z^r/\Lambda'_p\)$.
\end{proof}

For any positive real number $R$, denote $\mathcal{T}(R)$ be the set of primes at which $E$ has a bad reduction, together with the primes $p$ of good reduction with $\rho_{E,p} \le R$. We have the following analogue of Lemma~\ref{lem:LandW}.

\begin{lem}\label{lem:TRbound}
Suppose $E(\Q)$ has rank $r$. Then for any $R \ge 2$, we have
\[L(\cT(R))\ll R^{1+\frac{2}{r}}.
\]
\end{lem}

\begin{proof}
The proof of \cite[Proposition 5.4]{AGM10} implies that, for any prime $p \in \cT(R)$ we have $p \mid d_P$, for some
\[
P = n_1 G_1 + \cdots + n_r G_r,
\]
where $|n_i| \le R^{1/r}$ for each $i = 1,\ldots,r$. In particular, we can write
\[L(\cT(R))\le \sum_{\substack{P=n_1G_1 + \cdots + n_rG_r\\ \max\{|n_1|,\ldots, |n_r|\}\le R^{1/r}}} \log d_P.\]

On the other hand, we have $1\ll \h(P)$; see for instance \cite[Theorem VIII.9.10(a)]{Silv-Book}. In particular, we can write $\log d_P\ll \h(P)\ll R^{2/r}$. The proof concludes.
\end{proof}

\subsection{Proof of Theorem~\ref{thm:prodoverallpts}}
The proof is completely analogous to that of Theorem~\ref{thm:prod}, so we provide only a brief sketch. 

Let $1<R<H^{1/2}$ be a real number to be chosen suitably later, and first count the $s$-tuples $(P_1, \ldots, P_s)\in \cD_{\ell}(H)$, for which we have
$$\nu_p(d_{P_i})\not\equiv 0 \pmod \ell \implies p \in \cT(R),~\mathrm{for~some}~1\leq i\le s.$$

Lemma~\ref{lem:powerwithsunit} and Lemma~\ref{lem:TRbound} show that the number of such $s$-tuples is bounded by $$ H^{\frac{r_1+\ldots+r_s-r_{
\min}}{2}} \exp\(O\(R^{1+\frac{2}{r_{\min}}}\)\).$$

To count the remaining tuples $(P_1,\ldots, P_s)\in \cD_{\ell}(H)$, we consider the graph $\widehat{\cG}$ on such $s$ vertices $P_1, \ldots, P_s$, where the edges are defined by the condition 
\begin{equation*}
\nu_p(d_{P_i})\not\equiv 0 \pmod \ell \quad \mathrm{and}\quad \nu_p(d_{P_j})\not\equiv 0 \pmod \ell,
\end{equation*}
for some prime $p\not \in \cT(R)$.

As in the proof of Theorem~\ref{thm:prod}, $\widehat{\cG}$ has no isolated vertex. Therefore, by Lemma~\ref{lem:GraphCover}, there exists a subset $\cI$ of $\{1,\ldots, s\}$ with $m=\sharp\,  \cI\leq \fl{s/2}$, such that for any $j\in \{1,\ldots,s\}\setminus \cI$, there exists some $i\in \cI$, such that $P_i$ and $P_j$ satisfy \eqref{eqn:edge}. Then arguing exactly as in the proof of Theorem~\ref{thm:prod}, applying Lemma~\ref{lem:divbypoverallpointsbetter} (instead of Lemma~\ref{lem:divbyp}) and Lemma~\ref{lem:GraphCover}, the number of such $s$-tuples is 
\[O\(H^{\frac{r_1+\ldots+r_s}{2}}R^{-\rf{s/2}}\),\quad \mathrm{as}\quad R< H^{1/2}.
\]  

Therefore, we finally derive
$$\sharp\,\cD_{\ell}(H)\le \, H^{\frac{r_1+\ldots+r_s-r_{\min}}{2}} \exp\(O\(R^{1+\frac{2}{r_{\min}}}\)\)+O\(H^{\frac{r_1+\ldots+r_s}{2}}R^{-\rf{s/2}}\).$$

The proof concludes by taking $R=c(\log H)^{1/\(1+\frac{2}{r_{\min}}\)}$, for some sufficiently small constant $c>0$.
\qed

\section{Further applications of the underlying techniques}\label{sec:applications}

\subsection{Perfect powers in the single-variable shifted products}\label{sec:singlevariable}
Let us assume that $E_i=E$ for each $i=1,\ldots, s$, and let $P\in E(\Q)$ be a fixed non-torsion point. Consider the shifted products of the form 
\begin{equation*}
d_{(n+h_1)P} \cdots d_{(n+h_s)P},\quad n\in \mathbb{N}.
\end{equation*}

In this context, again recall that Hajdu, Laishram and Szikszai~\cite{HLS} have shown that, for any fixed $\ell\ge 2$, there are finitely many perfect $\ell$th powers among such products, uniformly over the shifts of the form
\[
\hsf = (h, 2h, \dots, sh), \quad h,s \in \mathbb{N}.
\]

In light of their result, for any fixed prime $\ell$ and any fixed integer $s\ge 2$, let us consider  
\begin{align*}
&S_{\ell,\Psf}(M,N,H)\\
&\quad\quad =\max_{\hsf \in \cH_s}\sharp\, \left\{ n\in   (M,M+N]: d_{(n+h_1)P}\cdots d_{(n+h_s)P} \in \cN_{\ell}\right\},
\end{align*}
where we denote 
$$\cH_s=\{(h_1,\cdots,h_s) \in (\Z\cap [-H,H])^s: h_i\ne h_j,~\mathrm{for}~1\le i\ne j\le s\}.$$
Assume that $2H<R$. Then for any prime $q>R$, we have
$$h_i\ne h_j \pmod q,\quad \mathrm{for~any}~1\le i\ne j\le s.$$

Using the same sieving argument as in the proof of Lemma~\ref{lem:sq+sqwithsunit}, one can easily derive the following estimate
\begin{align*}
&\max_{\hsf \in \cH_s}\sharp\, \left\{n\in (M,M+N]: d_{(n+h_i)P}\in \mathcal{O}^{*}_{\Q,\cW(R)}\cdot \cN_{\ell},~\for~i=1,\cdots,s\right\}\\
&\qquad\qquad \qquad\qquad\quad \ll N \, \(\frac{\log \, \max\{R,\log N\}}{\log  N}\)^s.
\end{align*}

For any remaining $n\in (M,M+N]$ with $d_{(n+h_1)P}\cdots d_{(n+h_s)P} \in \cN_{\ell}$, there exists a prime $p\not\in \cW(R)$ such that 
$$p\mid d_{(n+h_i)P}\quad \mathrm{and}\quad p\mid d_{(n+h_j)P}\quad \mathrm{for~some}~i\ne j.$$
In particular, we have
\begin{equation}\label{eqn:impcond}
\rho_p\mid h_i-h_j,\quad \mathrm{for~some}~i\ne j.
\end{equation}

However, \eqref{eqn:impcond} does not hold, as $\mathbf{h} \in \mathcal{H}_s$ and $2H < R$. Therefore, choosing $R=3H$, we finally derive the following result.

\begin{thm}\label{thm:singlevar}
Let $s,\ell\ge 2$ be fixed. Then uniformly over $M \ge 0$, we have
$$S_{\ell,\Psf}(M,N,H)\ll N \, \(\frac{\log \, \max\{H,\log N\}}{\log N}\)^{s}.$$
\end{thm}

Thus, $S_{\ell,\Psf}(M,N,H)$ is $o(N)$ over the range $1\le H\le \exp\(\frac{\log N}{G(N)}\)$, for any function $G(N)\to \infty$ as $N\to \infty$.

\subsection{Products of images of rational functions}\label{sec:CcapH}
Assume again that we are given $s$ not necessarily distinct elliptic curves $E_1, \ldots, E_s$ as in \eqref{eq:Weier} over $\Q$ and of positive ranks. Fix an $s$-tuple $\Psf=(P_1, \ldots, P_s)$ of non-torsion points $P_i\in  E_i(\Q)$, for each $i=1,\ldots, s$.

As in the proof of Theorem~\ref{thm:withx}, one can estimate the number of $s$-tuples of rational points whose product of the $x$-coordinates is a perfect $\ell$th power. Again note that, one of the ingredients is already available, due to Lemma~\ref{lem:divbyp} or Lemma~\ref{lem:divbypoverallpointsbetter}. Therefore, we only need an analogue of Lemma~\ref{lem:powerwithsunit}. For $\ell>2$, we have
\begin{equation*}
x(P)\in \mathcal{O}^{*}_{\Q,\cW(R)} \cdot \cQ_{\ell} \implies d_{P}\in \mathcal{O}^{*}_{\Q,\cW(R)} \cdot \cN_{\ell}.
\end{equation*}

For $\ell=2$, the relevant family of curves is 
\begin{equation*}
C'_u:y^2=u^3x^6+aux^2+b,\quad u\in \mathcal{O}^{*}_{\Q,\cW(R)}/(\mathcal{O}^{*}_{\Q,\cW(R)})^{2\ell},\end{equation*}
where $a,b$ are as in \eqref{eq:Weier}. Now we must assume that $b\ne 0$, as otherwise any $x(P)$ is a square up-to some factor of $a$. If $b\neq 0$, then each such $C'_u$ is non-singular and of genus $2$ (follows from~\cite[Exercise~2.14]{Silv-Book}). We may now apply \cite[Theorem~1.1]{R10} directly. Consequently, in all such cases, one obtains analogues of the results of Section~\ref{sec:results} for products of the $x$-coordinates of points of the form $nP+Q,~n\in (M,M+N]$, as well as for products taken over rational points of bounded canonical height.

Now let $f\in \Q(E)$ be a rational function. Mimicking the proof of Lemma~\ref{lem:divbyp} (or Lemma~\ref{lem:divbypoverallpointsbetter}), we can derive corresponding analogues of these lemmas, where the implied constant depends on $\deg(f)$, the degree of the rational map $f: E \to \mathbb{P}^1$.

On the other hand, to obtain an analogue of
Lemma~\ref{lem:powerwithsunit}, that is, an estimate for
\begin{equation*}
\sharp\, \{P\in E(\Q):
f(P)\in \mathcal{O}^{*}_{\Q,\cW(R)}\cdot \cQ_{\ell}\},
\end{equation*}
we first write $f=\frac{A(x)+yB(x)}{C(x)}$, where $A,B,C\in \mathbb{Z}[x]$ and $x,y$ are the
Weierstrass coordinate functions on $E$. Let us assume that
\begin{equation}\label{eqn:f}
2\deg A < 2\deg B + 3 \quad \text{or} \quad \deg A \ne \deg C,
\end{equation}
or in other words, we assume that $f$ has either a zero or a pole at $O$. 

Writing $P$ as in \eqref{eqn:P} with $a_P\ne 0$ and  $b_P \ne 0$ (note that only finitely many points $P$ fail to satisfy this), and expanding and clearing denominators, we obtain $f(P) = d_P^{k_f} \, r_P$, for some nonzero integer $k_f$ depending only on $f$, and some rational number $r_P$ whose numerator and denominator (in lowest terms) are both coprime to $d_P$. Therefore, we can again use Lemma~\ref{lem:powerwithsunit}, at least for any prime $\ell$ not dividing $k_f$. 

Thus, we obtain an analogue of Theorem~\ref{thm:withx} for any function $f$ satisfying \eqref{eqn:f}, rather than only $x$, and for any prime $\ell$ coprime to $k_f$. For instance, the function $f=y$ satisfies \eqref{eqn:f}, with $k_y=-3$.

\subsection{Sum of two squares among the products}\label{sec:sumofsquares}
Let us denote
\begin{align*}
D_{\square+\square,P,\cW(R)}(M,N)
   &=\sharp\, \bigl\{ n\in (M,M+N]: d_{nP} \\
   &\quad\quad\;\; \text{is a sum of two squares up to a}~\cW(R)\text{-}\mathrm{unit} \bigr\}.
\end{align*}

To estimate $D_{\square+\square,P,\cW(R)}(M,N)$ one can again argue as in Lemma~\ref{lem:sq+sqwithsunit}, adapting \cite[Theorem 1.1]{BLMN}. The only difference is that one now needs to sieve out the indices of denominators, modulo the new set of primes $\Lambda'(R)=\{q~\mathrm{prime}:2R<q\}\cap \Lambda'$, where
$$\Lambda'=\{q~\mathrm{prime}:d_{qP}~\text{is not a sum of two squares}\}.$$

Thanks to \cite[Proposition 3.19]{BLMN}, which ensures that $\Lambda'$ has a positive relative density $\omega=\omega(E,P)>0$ in the set of all primes, assuming that $P$ is in the connected component of $O$ in $E(\mathbb{R})$. Analogous to Lemma~\ref{lem:sq+sqwithsunit} and \cite[Theorem 1.1]{BLMN}, one obtains
\begin{equation}\label{eqn:[]+[]uptoW}
D_{\square+\square,P,\cW(R)}(M,N)
\ll N \left(\frac{\log \max\{R,\log N\}}{\log N}\right)^{\omega}.
\end{equation}

Again assume that we are given $s$ not necessarily distinct elliptic curves $E_1, \ldots, E_s$ as in \eqref{eq:Weier} over $\Q$ and of positive ranks. Fix an $s$-tuple $\Psf=(P_1, \ldots, P_s)$ of non-torsion points $P_i\in  E_i(\Q)$, for each $i=1,\ldots, s$. Then we denote
\begin{align*}
Y_{\square+\square,\Psf}(M,N) &= \sharp\, \{ (n_1, \ldots, n_s)\in (M,M+N]^s:\\
&\qquad\qquad y(n_1P_1) \cdots y(n_sP_s)~\text{is a sum of two squares} \}.
\end{align*}

For each $i=1,\ldots, s$, let $\omega_i=\omega(E_i,P_i)$ is as in \cite[Theorem~1.1]{BLMN}, whenever $P_i$ is in the connected component of $O_i$ of $E_i(\R)$. Moreover, let $\omega(\Psf)$ be the minimum of all such $\omega_i$. Then arguing exactly as in the proof of \eqref{eqn:generalM}, we derive the following result.

\begin{thm}\label{thm:y=[]+[]}
Let $s\ge 2$ be fixed. Then uniformly over $M \ge 0$, we have
    $$Y_{\square+\square,\Psf}(M,N)\ll N^s\( \frac{\log \log N}{\log N}\)^{c(\Psf) \, \omega(\Psf)},$$
    where $c(\Psf)$ is the number of indices $i=1,\ldots, s$ such that $P_i$ is in the connected component of $O_i$ of $E_i(\R)$.
\end{thm}

In particular, Theorem~\ref{thm:y=[]+[]} extends~\cite[Theorem~1.1]{BLMN} to $s\ge 2$.

\section{Concluding remarks}
\subsection{An illustration of the admissible exponents}\label{sec:admexponents} 
Note that the conclusion of Theorem~\ref{thm:prod} also holds for equations of the form
$$
d_{n_1P_1 + Q_1}^{k_1} \cdots d_{n_sP_s+Q_s}^{k_s} \in \cN_{m},
$$
provided that there exists a prime $\ell \mid m$, which does not divide $k_i$ for any $i=1,\ldots,s$. It is now natural to ask how many $(s+1)$-tuples $(k_1,\ldots,k_{s},m)$ of integers do not satisfy such a condition, with $|k_i|\le N$ for each $i=1,\ldots,s$ and $2\le m\le N$. For a parameter $2\le y\le N$ to be chosen suitably, a crude estimation is given by 
\begin{equation}\label{eqn:expsbound}
N^{s}\Psi(N,y)+N^{s+1}\sum_{\substack{y< \ell\\ \ell~\mathrm{prime}}} 
\frac{1}{\ell^2}\ll N^{s}\Psi(N,y)+\frac{N^{s+1}}{y \log y},
\end{equation}
where again, $\Psi(N,y)$ denotes the number of $y$-smooth positive integers up to $N$. In particular one easily sees that the quantity in \eqref{eqn:expsbound} is $o(N^{s+1})$, for instance by \eqref{eqn:psinrbound} and taking $y$ as in \eqref{eqn:newR}.

Note that Theorem~\ref{thm:prod} also applies to the $(s+1)$-tuple of exponents $(k_1,\ldots,k_s,m)$, for which there exists a prime $\ell \mid m$ such that $\ell \nmid k_i$ for at least one $1\le i \le s$. Indeed, let $t$ be the number of such indices $1\le i\le s$ such that $\ell \nmid k_i$. Then by trivially fixing $n_j\in (M,M+N]$ for each of the remaining $s-t$ indices $j$, we save at least $(\log N)^{\lceil t/2 \rceil/3}$ for $t\ge 2$ and $N$ when $t=1$, compared to the trivial bound of order $N^s$.

\subsection{Difficulties in obtaining a power saving}\label{sec:bottleneck}
An interesting open problem is whether the savings obtained in Theorem~\ref{thm:prod} can be substantially improved. Of course, when $s=1$, there are only $O(1)$ such values of $n_1$. This indicates that one should hope for a power saving. The bottleneck arises from the bound for $D_{\ell,\cW}$ in Lemma~\ref{lem:powerwithsunit}. Perhaps one cannot generally do better than $\exp(O(\sharp\, \cW))$, because there are $\mathcal{O}^{*}_{\Q,\cW}/(\mathcal{O}^{*}_{\Q,\cW})^{2\ell}$ equations to consider in \eqref{eqn:allequations}. Moreover, since $1$ is trivially an $\ell$th power, the Siegel-type bounds (see \cite[Theorem 3.6, Chapter~IX]{Silv-Book}) perhaps suggest that such an estimate is essentially best possible.

However, even assuming a bound of order $\exp(O(\sharp\,\cW))$ in Lemma~\ref{lem:powerwithsunit}, we only save an additional power of $\log\log N$ in Theorem~\ref{thm:prod} and a power of $\log \log H$ in Theorem~\ref{thm:prodoverallpts}. Nevertheless, it may be of independent interest in its own right to study the family of curves $C_{u,v,\ell}$ defined in \eqref{eqn:Cuvell}. Motivated by this, we pose the following question.
\begin{question}
    Let $\cW_K$ and $\cW$ be as in Lemma~\ref{lem:C'u}. Do we have
  \[\sum_{u,v \in \mathcal{O}^{*}_{K,\cW_K}/(\mathcal{O}^{*}_{K,\cW_K})^{2\ell}}\sharp\,C_{u,v,\ell}(K)   \le \exp\(O(\sharp\,\cW)\),
  \]
where the implied constant depends only on $\ell$ and $K$? 
\end{question}

It is clear that our argument gives a power saving in \eqref{eqn:M0}, when Lemma~\ref{lem:sq+sqwithsunitm0} saves a power with $R=N^{\varepsilon}$, for some fixed $\varepsilon>0$. This requires a better understanding of $d_{nP}$ when $n$ is $N^{\varepsilon}$-smooth but has a moderate prime factor; for example, in the range $\big((\log N)^K,N^{\varepsilon}\big)$ for some fixed and large $K$. On the other hand, by Silverman’s primitive divisor theorem, for any integer $1\le \Omega<1/3\varepsilon$, we have
\[
\sharp\, \{n\in\mathbb{N}: d_{nP}\in \mathcal{O}^{*}_{\Q,\cW\(N^{\Omega\varepsilon}\)}\}\ll N^{3\Omega \varepsilon}.
\]

So for any $n$ not as above, we get a divisor $r_p>N^{\Omega\varepsilon}$ of $n$. Then $n$ is not $N^{\varepsilon}$-smooth, when $n$ has at most $\Omega$ prime factors, counted with multiplicity. Thus, Lemma~\ref{lem:sq+sqwithsunitm0} is considerably simpler in such cases.

A recent preprint of Kym~\cite{Kym26} studies $s$-tuples of integers $(n_1,\ldots,n_s)$ for which
\[
d_{n_1P}\cdots d_{n_sP}\notin \mathcal{N}_{\ell},
\]
for a fixed prime $\ell$; see, for example, Theorem~5.5 and the subsequent results, including Corollary~5.6, Proposition~5.7, and Corollary~5.8 therein. These results are proved under some additional conditions, such as the divisibility of $d_P$ by $2$ or $3$, or under the weaker Hypothesis~4.4 of \cite{Kym26}. As already mentioned in Section~\ref{sec:strats}, our Lemma~\ref{lem:powerwithsunit} removes the need for this additional hypothesis, since it gives an unconditional version of \cite[Proposition~4.1]{Kym26}.

\subsection{A power saving under the abc-conjecture}\label{sec:abc}
Recall the well-known $abc$-conjecture (over $\Q$) of Masser and Oesterlé, as mentioned in \cite[Page 226]{Silv88}. Then the proof of \cite[Theorem 2]{Silv88} gives the following.

\begin{lem}\label{lem:strongerzsigmondy}
    Assume that the abc-conjecture over $\Q$ is true and the $j$-invariant of $E$ is either $0$ or $1728$. Then there exists a constant $C(E,P)>0$ such that for any $n>C(E,P)$, there exists a primitive prime divisor $p$ of $d_{nP}$, such that $\nu_p(d_{nP})=1$.
\end{lem}

Now assume the hypotheses of Lemma~\ref{lem:strongerzsigmondy}. Let $s,\ell \ge 2$ be fixed. Then as a consequence of Lemma~\ref{lem:strongerzsigmondy}, uniformly over $M\ge 0$, we have 
\begin{equation}\label{eqn:Nsaving}
\sharp\, \cD_{\ell,\Psf,\isf}(M,N) \ll N^{s-1}.
\end{equation}

The proof of \eqref{eqn:Nsaving} is straightforward. Just note that for any fixed tuple $(n_2,\ldots, n_{s})\in (M,M+N]^{s-1}$, we have
    $$\sharp\left\{n_1\in \N:d_{n_1P_1}^{k_1}\in d_{n_2P_2}^{-k_2} \cdots d_{n_sP_s}^{-k_s}\cdot \cN_{\ell}\right\}\le C(E,P)+1.$$
Indeed, Lemma~\ref{lem:strongerzsigmondy} implies that for any such fixed tuple $(n_2,\ldots, n_{s})$, we have at most one $n_1>C(E,P)$ satisfying $d_{n_1P_1}^{k_1}\in d_{n_2P_2}^{-k_2} \cdots d_{n_sP_s}^{-k_s}\cdot \cN_{\ell}$, since we assume that $k_1\in [1,\ell]$.    

Furthermore, a generalisation of Lemma~\ref{lem:strongerzsigmondy} dropping the condition on the $j$-invariant is possible, using Kühn and Müller's~\cite[Theorem~1.1]{KM15}, under the assumption of the \textit{uniform abc-conjecture}~\cite[Conjecture~2.2]{KM15}. Note that the $abc$-conjecture over $\Q$ is a special case of the uniform $abc$-conjecture, as mentioned in \cite[Remark~2.3]{KM15}.

It is now natural to ask whether one can obtain a power saving for $\sharp\, \cD_{\ell,\Psf,\Qsf}(M,N)$, for any arbitrary $s$-tuple of rational points $\Qsf$. In this setting, the primitive divisor theorem for the sequence $d_{nP+Q}$ is currently known only when $Q$ is a torsion point, thanks to Verzobio~\cite{Ver20,Ver21,Ver23}. However, to obtain a power saving for $\sharp\, \cD_{\ell,\Psf,\Qsf}(M,N)$, it suffices to have a good lower bound on the \textit{powerless part} of $d_{nP+Q}$.

Assuming the $abc$-conjecture over $\Q$, such a lower bound can be extracted from the proof of \cite[Theorem~2]{Silv88} in the case where the $j$-invariant is $0$ or $1728$, with \cite[Lemma~3.1]{BBHOS26} playing a crucial role in the argument. To treat the cases of arbitrary $j$-invariants, one may again use the uniform $abc$-conjecture.

\section*{Acknowledgements}

The author sincerely thanks Philip Habegger for suggesting useful tools for bounding the rank of Jacobians, and Alina Ostafe for bringing \cite{DGH21} to the author’s attention, which led to a simplification and improvement of a preliminary bound. He is also very grateful to Igor E. Shparlinski for reading the manuscript, providing many helpful remarks, and drawing attention to \cite{BKS21}, which is one of the key motivations for this work. The author also sincerely thanks Simon L. Rydin Myerson for many helpful suggestions. Further thanks go to Lajos Hajdu and Shanta Laishram for related discussions. Some parts of this work have been carried out during a visit to ISI Delhi, and the author is grateful to Shanta Laishram for the hospitality.

The author has been supported by the Australian Research Council Grant DP230100530.

\end{document}